\documentclass[a4paper,11pt]{amsart}

\usepackage{amsmath,amssymb}
\usepackage{mathtools}
\usepackage{mathrsfs}
\usepackage{stmaryrd}
\usepackage[pagebackref,
    ,pdfborder={0 0 0}
    ,urlcolor=black,a4paper,hypertexnames=false]{hyperref}
\hypersetup{pdfauthor={Clara L\"oh, Giovanni Sartori},pdftitle=Integral foliated simplicial volume and ergodic decomposition}

\usepackage[initials, backrefs]{amsrefs}
\makeatletter
\def\print@backrefs#1{%
    \space\SentenceSpace{\color{black!60}Cited on page:~\csname br@#1\endcsname}%
}
\makeatother

\usepackage{pgf}
\usepackage{tikz}
\usetikzlibrary{decorations.pathmorphing,calc}
\tikzstyle{every picture}+=[remember picture]

\usepackage{tikz-cd}

\newtheorem{thm}{Theorem}[section]
\newtheorem{prop}[thm]{Proposition}
\newtheorem{lem}[thm]{Lemma}
\newtheorem{cor}[thm]{Corollary}

\newtheorem{question}[thm]{Question}

\theoremstyle{remark}
\newtheorem{rem}[thm]{Remark}

\theoremstyle{definition}
\newtheorem{defi}[thm]{Definition}

\usepackage{amsfonts}
\newcommand{\Z}{\mathbb{Z}}
\newcommand{\Q}{\mathbb{Q}}
\newcommand{\R}{\mathbb{R}}
\newcommand{\N}{\mathbb{N}}

\DeclareMathOperator{\id}{id}

\DeclareMathOperator{\map}{map}

\DeclareMathOperator{\Linf}{L^\infty}
\DeclareMathOperator{\B}{B}

\DeclareMathOperator{\spanz}{Span_\Z}
\DeclareMathOperator{\fc}{FC}

\DeclareMathOperator{\Erg}{Erg}
\DeclareMathOperator{\Prob}{Prob}

\DeclareMathOperator{\ess}{ess}

\def\fa#1{%
  \forall_{#1}\;\;\;}

\def\args{\;\cdot\;}

\def\longrightarrow{\rightarrow}
\def\longmapsto{\mapsto}
\renewcommand{\epsilon}{\varepsilon}
\renewcommand{\phi}{\varphi}
\def\subseteq{\subset}
\def\varnothing{\emptyset}

\def\ucov#1{%
  \widetilde{#1}
}

\def\symmdiff{\mathbin{\triangle}}

\def\linfz#1{\Linf(#1,\Z)}
\def\bz#1{\B(#1,\Z)}
\def\nz#1{\text{N}(#1,\Z)}

\def\actson{\curvearrowright}

\def\qor{\quad\text{or}\quad}

\makeatletter
\newcommand\norm{\bBigg@{0.8}}
\makeatother

\newcommand{\ifsv}[2][norm]{\csname #1l\endcsname\bracevert\!#2\!%
                            \csname #1r\endcsname\bracevert}
\newcommand{\ifsvp}[3][norm]{\csname #1l\endcsname\bracevert\!#2\!%
                            \csname #1r\endcsname\bracevert\!^{#3}}

\newcommand{\essn}[2][norm]{\csname #1l\endcsname\vert #2 \csname#1r\endcsname\vert_{1,\ess}}

\DeclarePairedDelimiter{\abs}{\lvert}{\rvert}

\usepackage{color}
\usepackage{pdfcolmk}

\def\draftinfo{}

\author[C.~L\"oh]{Clara L\"oh}
\address{Fakult\"at f\"ur Mathematik\\
         Universit\"at Regensburg\\
         93040~Regensburg\\
         }
\email{clara.loeh@mathematik.uni-r.de}

\author[G.~Sartori]{Giovanni Sartori}
\address{Maxwell Institute and Department of Mathematics\\
         Heriot-Watt University\\
          Edinburgh~EH14~4AS\\
         }
\email{gs2057@hw.ac.uk}

\title[Integral foliated simplicial volume and ergodic decomposition]{Integral foliated simplicial volume\\ and ergodic decomposition}
\date{\today.\ \copyright{\ C.~L\"oh, G.~Sartori 2022}. 
    This work was supported by the CRC~1085 \emph{Higher Invariants} 
    (Universit\"at Regensburg, funded by the DFG) and is partially
    based on GS's MSc project.  
    \draftinfo\\
     MSC~2010 classification: 55N10, 55N35, 28D15} 
\begin{document}

\begin{abstract}
  We establish an integration formula for integral
  foliated simplicial volume along ergodic decompositions.
  This is analogous to the ergodic decomposition formula
  for the cost of groups. 
\end{abstract}

\maketitle
\vspace{-\baselineskip}

\section{Introduction}

The integral foliated simplicial volume is a dynamical version of the
simplicial volume of manifolds: It measures the size of fundamental
cycles of a manifold~$M$ with respect to twisted coefficients
in~$\linfz {X,\mu}$, where $\pi_1(M) \actson (X,\mu)$ is a probability
measure preserving action on a standard Borel probability space (see
Section~\ref{sec:ifsv} for the definitions). The integral foliated
simplicial volume provides upper bounds for the $L^2$-Betti
numbers~\citelist{\cite{gromov_metric}*{p.~305f}\cite{mschmidt}} and
the cost of the fundamental group~\cite{loeh_cost}.

In the case of a residually finite fundamental group, a dynamical system
of particular interest is the profinite completion. The cost of the
action on the profinite completion is the rank
gradient~\cite{abertnikolov}. Analogously, the integral foliated
simplicial volume with respect to the profinite completion equals the
stable integral simplicial
volume~\cite{loehpagliantini}*{Remark~6.7}. The stable integral
simplicial volume is one of the few known upper bounds for logarithmic
torsion homology growth~\cite{FLPS}*{Theorem~1.6}. However, as in the
fixed price problem for the cost of groups, it is unknown in general
whether different essentially free dynamical systems can lead to
different values for the same manifold; in particular, which dynamical
systems can be used for logarithmic torsion growth estimates?

Following the theory for cost~\cite{kechris_global}*{Corollary~10.14},
it has been shown that the integral foliated simplicial volume is
monotonic with respect to weak containment of dynamical
systems~\cite{FLPS}*{Theorem~1.5} and that for several classes of
manifolds all essentially free dynamical systems lead to the same
value~\citelist{\cite{FLPS}*{Theorem~1.9}\cite{fauserS1}\cite{loehmoraschinisauer}}.

In the present paper, in analogy with the ergodic decomposition
formula for the cost of
groups~\cite{kechrismiller}*{Proposition~18.4}, we show in
Section~\ref{sec:proof}:

\begin{thm}\label{thm:ifsvergdec}
  Let $M$ be an oriented closed connected manifold with fundamental
  group~$\Gamma$, let $(\alpha,\mu)\colon\Gamma \actson X$ be a
  standard probability action, and let $\beta \colon X \longrightarrow
  \Erg(\alpha)$ be an ergodic decomposition of~$(\alpha,\mu)$. Then
  \[ 
  \ifsvp{M}{(\alpha,\mu)}
  = \int_X \ifsvp{M}{(\alpha,\beta_x)} \;d\mu(x).
  \]
\end{thm}

In particular, there exists an ergodic parameter space
that realises the integral foliated simplicial volume:

\begin{cor}
  Let $M$ be an oriented closed connected manifold with fundamental
  group~$\Gamma$.  Then there exists an essentially free ergodic
  standard $\Gamma$-space~$(\alpha,\mu)$ with
  \[ \ifsv M = \ifsvp M {(\alpha,\mu)}.
  \]
\end{cor}
\begin{proof}
  Taking products of standard actions shows that there is an
  essentially free standard probability action~$(\alpha,\mu)\colon \Gamma
  \actson X$ with $\ifsv M = \ifsvp M
  {(\alpha,\mu)}$~\cite[Corollary~4.14]{loehpagliantini}.
  Let $\beta \colon X \longrightarrow \Erg(\alpha)$ be an ergodic
  decomposition for~$(\alpha,\mu)$, as provided by the ergodic
  decomposition theorem~\cite{varadarajan} (Theorem~\ref{thm:ergdec}).
  The ergodic decomposition formula for integral foliated simplicial volume
  (Theorem~\ref{thm:ifsvergdec}) gives
  \[ \ifsv M
  = \ifsvp M {(\alpha,\mu)}
  = \int_X \ifsvp M {(\alpha,\beta_x)} \;d\mu(x).
  \]
  Moreover, by definition, $\ifsvp M {(\alpha,\beta_x)} \geq \ifsv M$ for
  all~$x \in X$.  Therefore, we obtain
  \[ \ifsv M = \ifsvp M {(\alpha,\beta_x)}
  \]
  for $\mu$-almost every~$x \in X$. As $\alpha$ is essentially free,
  there also exists an~$x \in X$ such that the $\Gamma$-action on~$X$
  is essentially free with respect to~$\beta_x$ and simultaneously 
  $\ifsv M = \ifsvp M {(\alpha,\beta_x)}$ is satisfied
  (Remark~\ref{rem:essfree}).  By construction, $\beta_x$ is ergodic.
\end{proof}

For the proof of Theorem~\ref{thm:ifsvergdec}, we first show that for
each standard action~$\alpha \colon \Gamma \actson X$, there is a
\emph{countable} subcomplex~$\Sigma_*(M,X;\Z)$ of the (strict) chain
complex~$\bz X \otimes_{\Z \Gamma} C_*(\ucov M;\Z)$ with the following
property: For \emph{every} $\Gamma$-invariant probability
measure~$\nu$ on~$X$, we have
\begin{align*}
  \ifsvp M {(\alpha,\nu)}
= \inf
  \bigl\{ \ifsvp c {(\alpha,\nu)}
  \bigm|\; & c \in \Sigma_*(M,X;\Z) \\
           & \text{is a fundamental cycle}
  \bigr\}.
\end{align*}
The ergodic decomposition formula can then be viewed as an
instance of switching this particular infimum with integration.

A weaker result relating integral foliated simplicial volume and
ergodic parameter spaces was already known:

\begin{prop}[\protect{\cite[Proposition 4.17]{loehpagliantini}}]
  Let $M$ be an oriented closed connected manifold with fundamental
  group~$\Gamma$. Then, for every~$\varepsilon \in \R_{>0}$, there exists
  an ergodic standard $\Gamma$-space~$(\alpha,\mu)$ such that
  \[
  \ifsvp M {(\alpha,\mu)}\le\ifsv M+\varepsilon.
  \]
\end{prop}

The following terminology is borrowed from the theory of
cost~\cite{gaboriau_cost}:

\begin{defi}[fixed price~\cite{loeh_cost}*{Definition~1.3}]
  Let $M$ be an oriented closed connected manifold; we say that $M$ has 
  \emph{fixed price} if $\ifsvp M {(\alpha,\mu)} = \ifsvp M {(\alpha',\mu')}$
  holds for all essentialy free standard
  $\Gamma$-spaces~$(\alpha,\mu)$ and~$(\alpha',\mu')$.
\end{defi}

As in the case of cost the fixed price problem for manifolds is still open:

\begin{question}[fixed price problem~\protect{\cite{FLPS}*{Question~1.13}}]
  Do all oriented closed connected manifolds have fixed price?
\end{question}

\subsection{Organisation of this paper}

We recall the definition of integral foliated simplicial volume in
Section~\ref{sec:ifsv}; ergodic decompositions are recalled in
Section~\ref{sec:ergdec}. The proof of Theorem~\ref{thm:ifsvergdec} is
given in Section~\ref{sec:proof}.

\subsection*{Acknowledgements}

We would like to thank the anonymous referee for the constructive
report and for catching several inaccuracies and a major omission
in the main argument.

\section{Integral foliated simplicial volume}\label{sec:ifsv}

Parametrised simplicial volume and integral foliated simplicial
volume arise as a dynamical generalisation of integral and real
simplicial volume: One replaces the constant integral/real
coefficients with twisted coefficients of spaces of (essentially)
bounded integer-valued 
functions~\citelist{\cite{gromov_metric}*{p.~305f}\cite{mschmidt}}.

\subsection{Basic definitions}

A \emph{standard Borel space} is a measurable space that is isomorphic
to a Polish space together with its Borel
$\sigma$-algebra~\cite{kechris2012classical}.

\begin{defi}[standard (probability) actions and bounded functions]
  \hfil
  \begin{itemize}
  \item A \emph{standard action} is a measurable action
    of a countable group on a standard Borel space.
  \item A \emph{standard probability action} is a
    pair~$(\alpha, \mu)$, where $\alpha \colon \Gamma \actson X$
    is a standard action and where $\mu$ is an $\alpha$-invariant
    probability measure on~$X$.
  \item If $(\alpha,\mu) \colon \Gamma \actson X$ is a standard
    probability action, then we equip the $\Z$-module~$\linfz {X,\mu}$
    of $\mu$-equivalence classes of measurable $\mu$-essentially
    bounded functions~$X \to \Z$ with the $\Z\Gamma$-module structure
    \begin{align*}
      \linfz {X,\mu} \times \Gamma & \to \linfz {X,\mu}
      \\
      (f,\gamma) & \mapsto
      \bigl(x \mapsto f(\gamma \cdot x)\bigr).
    \end{align*}
    The resulting $\Z\Gamma$-module is denoted by~$\linfz
    {\alpha,\mu}$.
  \end{itemize}
\end{defi}

In the following, in the notation of functions spaces etc.\
we will always use ``$\alpha$'' instead of~``$X$'' to emphasise
the underlying action instead of the underlying measure space.

Let $M$ be a connected manifold with fundamental group~$\Gamma$ and
universal covering~$\pi \colon \ucov M \to M$.  If $A$ is a right
$\Z\Gamma$-module, then we denote the twisted singular chain complex
and the twisted singular homology of~$M$ with coefficients in~$A$ by
(where $\Gamma$ acts by deck transformations on the singular simplices
of~$\ucov M$):
\begin{align*}
  C_*(M;A) & := A \otimes_{\Z\Gamma} C_*(\ucov M;\Z) \\
  H_*(M;A) & := H_*\bigl( C_*(M;A)\bigr).
\end{align*}
For the constant coefficients~$\Z$, the universal covering map
induces a $\Z$-chain isomorphism between the untwisted
singular chain complex of~$M$ with $\Z$-coefficients
and~$\Z \otimes_{\Z\Gamma} C_*(\ucov M;\Z)$. We will always
use this identification for~$C_*(M;\Z)$.

\begin{defi}[parametrised fundamental cycles]
  Let $M$ be an oriented closed connected manifold with fundamental
  group~$\Gamma$, let $(\alpha,\mu) \colon \Gamma \actson X$
  be a standard probability action. We write
  \[ i_M^{(\alpha,\mu)}\colon C_*(M;\Z)
     \to C_*\bigl(M;\linfz{\alpha,\mu}\bigr)
  \]
  for the chain map induced by the inclusion of~$\Z$ into~$\linfz
  {\alpha,\mu}$ as constant functions. 
  All cycles in~$C_*(M;\linfz{\alpha,\mu})$ representing
  \[ 
  [M]^{(\alpha,\mu)}:=H_*(i_M^{(\alpha,\mu)})\bigl([M]_\Z\bigr)
  \in H_*\bigl(M ; \linfz{\alpha,\mu} \bigr)
  \]
  are called \emph{$(\alpha,\mu)$-parametrised fundamental
  cycles of~$M$}.
\end{defi}

\begin{defi}[integral foliated simplicial volume]
  Let $M$ be an oriented closed connected manifold with fundamental
  group~$\Gamma$ and let $(\alpha,\mu)\colon\Gamma\actson X$ be a
  standard probability action.
  \begin{itemize}
  \item A chain $\sum_{j=1}^mf_j \otimes {\sigma}_j\in
    C_*(M;\linfz{\alpha,\mu})$ is in \emph{reduced form} if for
    all~$j,k\in\{1,\dots,m\}$ with $j\ne k$ we have that
    $\pi\circ{\sigma}_j\ne \pi\circ{\sigma}_k$. In other words the
    singular simplices ${\sigma}_1,\dots,{\sigma}_m$ of~$\ucov{M}$
    arise from different simplices in~$M$. Reduced forms are essentially
    unique (up to the $\Gamma$-action on the simplices).
  \item Let $c=\sum_{j=1}^mf_j\otimes {\sigma}_j\in
    C_*(M;\linfz{\alpha,\mu})$ be in reduced form.
    The \emph{$(\alpha,\mu)$-parametrised $\ell^1$-norm} of~$c$ is
    \[
    \ifsvp{c}{(\alpha,\mu)}=\sum_{j=1}^m\int_X\abs{f_j}\;d\mu\in\R_{\geq 0}.
    \]
  \item The \emph{$(\alpha,\mu)$-parametrised simplicial volume}
    of~$M$ is the infimum
    \begin{align*}
      \ifsvp{M}{(\alpha,\mu)}
      :=\inf\bigl\{\ifsvp{c}{(\alpha,\mu)}
      \bigm|
      & \;\text{$c\in C_*(M;\linfz{\alpha,\mu})$ is an }\\
      & \;\text{$(\alpha,\mu)$-parametrised fundamental cycle}\bigr\}.
    \end{align*}
  \item The \emph{integral foliated simplicial volume}~$\ifsv M$
    of~$M$ is the infimum of the~$\ifsvp{M}{(\alpha,\mu)}$ over all
    isomorphism classes of standard probability
    actions~$(\alpha,\mu)\colon\Gamma \actson X$.
  \end{itemize}
\end{defi}

	If $\zeta\in H_*(M;\linfz{\alpha,\mu})$, then we denote 
	\[
	\ifsvp{\zeta}{(\alpha,\mu)}:=\inf\bigl\{\ifsvp{c}{(\alpha,\mu)}\bigm|\text{$c\in C_*(M;\linfz{\alpha,\mu})$ and $[c]=\zeta$}\bigr\}
	\]
	so that we can in particular express the $(\alpha,\mu)$-parametrised simplicial volume of $M$ as $\ifsvp{[M]^{(\alpha,\mu)}}{(\alpha,\mu)}$.

\begin{prop}[comparison with integral and real simplicial
    volume~{\cite[Proposition 4.6]{loehpagliantini}}]
  \label{prop:comparison}
  Let $M$ be an oriented closed connected manifold with fundamental
  group~$\Gamma$. For every standard probability $\Gamma$-action~$(\alpha,\mu)$,
  we have
  \[
  \lVert M\rVert_{\R}\le\ifsv M\le\ifsvp M {(\alpha,\mu)}\le\lVert M\rVert_{\Z},
  \]
  where $\lVert M\rVert_\R$ and $\Vert M \rVert_\Z$ denote the
  real and integral simplicial volume, respectively.
\end{prop}

Computations of integral foliated simplicial volume have been
performed for various oriented closed connected aspherical
manifolds~\cite{mschmidt,loehpagliantini,FLPS,fauserS1,campagnolocorro,FLMQ,loehmoraschinisauer}.

\subsection{A strict version}

For the proof of Theorem~\ref{thm:ifsvergdec}, we will
need to consider chains and norms with respect to different
probability measures on the same measurable action. In order
to avoid complications caused by sets of measure~$0$ with
respect to different measures, we explain how to compute
the integral foliated simplicial volume via strict chains.

\begin{defi}[bounded functions]
  If $\alpha \colon \Gamma \actson X$ is a standard
  action, we equip the $\Z$-module~$\bz X$ of measurable
  bounded functions~$X \to \Z$ with the induced right
  $\Z\Gamma$-module structure given by
  \begin{align*}
    \bz X \times \Gamma & \to \bz X
    \\
    (f,\gamma) & \mapsto
    \bigl(x \mapsto f(\gamma \cdot x)\bigr).
  \end{align*}
  This $\Z\Gamma$-module is denoted by~$\bz \alpha$.
\end{defi}

If $(\alpha,\mu) \colon \Gamma \actson X$ is a standard probability
action, then there is a canonical isomorphism~$\linfz {\alpha,\mu}
\cong \bz \alpha/\nz{\alpha,\mu}$ of $\Z\Gamma$-modules, where
$\nz{\alpha,\mu} \subset \bz{\alpha}$ is the $\Z\Gamma$-submodule of
functions that are $\mu$-almost everywhere~$0$. Moreover, as bounded
functions to~$\Z$ only take on finitely many different values, $\bz
{\alpha}$ is generated as a $\Z$-module by the set~$\{\chi_A \mid A
\subset X \text{ measurable}\}$.

\begin{defi}[parametrised strict fundamental cycles]
  Let $M$ be an oriented closed connected manifold with fundamental
  group~$\Gamma$ and let $\alpha \colon \Gamma \actson X$ be a
  standard action. We write
  \[ j_M^{\alpha} \colon C_*(M;\Z) \to C_*\bigl(M;\bz \alpha \bigr)
  \] 
  for the chain map induced by the inclusion of~$\Z$ into~$\bz \alpha$
  as constant functions. All cycles in~$C_*(M;\bz \alpha)$ representing
  \[ H_*(j_M^\alpha)\bigl([M]_\Z\bigr)
     \in H_*\bigl(M;\bz \alpha\bigr)
  \]
  are called \emph{$\alpha$-parametrised strict fundamental cycles
  of~$M$}.
\end{defi}

In the same way as for $L^\infty$-coefficients, we can introduce
a notion of chains to be in reduced form in~$C_*(M;\bz \alpha)$. 

\begin{defi}[$\ell^1$-norm for strict chains]
  Let $M$ be an oriented closed connected manifold with fundamental
  group~$\Gamma$ and let $(\alpha,\mu)\colon\Gamma\actson X$ be a standard
  probability action. If $c = \sum_{j=1}^m f_j\otimes \sigma_j \in C_*(M;\bz{\alpha})$
  is in reduced from, then we define
  \[
  \ifsvp{c}{(\alpha,\mu)} := \sum_{j=1}^m\int_X\abs{f_j}\;d\mu\in\R_{\geq 0}.
  \]
\end{defi}

\begin{prop}[integral foliated simplicial volume via strict chains]\label{prop:strict}
  Let $M$ be an oriented closed connected manifold with
  fundamental group~$\Gamma$ and let $(\alpha,\mu) \colon \Gamma
  \actson X$ be a standard probability action. Then
  \[ \ifsvp M {(\alpha,\mu)}
  = \inf
  \bigl\{ \ifsvp c {(\alpha,\mu)}
  \bigm| c \in C_*(M;\bz \alpha)
  \text{\ is a strict fundamental cycle}
  \bigr\}.
  \]
\end{prop}
\begin{proof}
  The canonical map~$\varphi \colon \bz \alpha \to \linfz{\alpha,\mu}$
  acts as identity on constant functions. The induced chain
  map
  \[ \Phi_* \colon C_*\bigl(M;\bz\alpha\bigr)
  \to C_*\bigl(M;\linfz{\alpha,\mu}\bigr)
  \]
  satisfies~$i_M^{(\alpha,\mu)} = \Phi_* \circ j_M^\alpha$ and thus
  maps $\alpha$-parametrised strict fundamental cycles (in reduced
  form) to $(\alpha,\mu)$-parametrised fundamental cycles (in reduced
  form). Moreover, $\Phi_*$ is isometric with respect to~$\ifsvp \cdot
  {(\alpha,\mu)}$.  Therefore, the estimate~``$\leq$'' of the claim
  holds.
  
  For the converse estimate, we argue via the approximation of
  boundaries: Let $n := \dim M$ and let $c \in
  C_n(M;\linfz{\alpha,\mu})$ be an $(\alpha,\mu)$-parametrised
  fundamental cycle.  It suffices to find an $\alpha$-parametrised
  strict fundamental cycle having norm at most~$\ifsvp c {(\alpha,\mu)}$. 
  By definition, there exist a fundamental
  cycle~$c_\Z \in C_*(M;\Z)$ of~$M$ and a chain~$d \in
  C_{n+1}(M;\linfz{\alpha,\mu})$ with
  \[ c = c_\Z + \partial d \in C_n\bigl(M ;\linfz{\alpha,\mu}\bigr).
  \]
  With $\varphi$ also $\Phi_*$ is surjective in every degree. Let
  $\widetilde d \in C_{n+1}(M;\bz \alpha)$ be a chain
  with~$\Phi_{n+1}(\widetilde d) = d$.
  For the strict chain 
  \[ \widetilde c:= c_\Z + \partial \widetilde d \in C_n\bigl(M;\bz \alpha\bigr)
  \]
  we obtain the estimate

\begin{align*}
	\ifsvp{\tilde c}{(\alpha,\mu)}	&=\ifsvp{\Phi_n(\tilde c)}{(\alpha,\mu)}						&\text{($\Phi_*$ is isometric)}\\
									&=\ifsvp{c_{\Z}+\Phi_n(\partial\tilde d)}{(\alpha,\mu)}			&\text{($i_M^{(\alpha,\mu)} = \Phi_* \circ j_M^\alpha$)}\\
									&=\ifsvp{c_{\Z}+\partial\Phi_{n+1}(\tilde d)}{(\alpha,\mu)}		&\text{($\Phi_*$ is a chain map)}\\
									&=\ifsvp{c_{\Z}+\partial d}{(\alpha,\mu)}						&\text{($\Phi_{n+1}(\tilde d)=d$)}\\
									&=\ifsvp{c}{(\alpha,\mu)} &
\end{align*}
  By construction, $\widetilde c$ is an $\alpha$-parametrised
  strict fundamental cycle of~$M$.
  This completes the proof of the estimate~``$\geq$'' of the claim.
\end{proof}

\begin{rem}[decomposition into invariant subspaces]
  \label{rem:decompinv}
  Let $M$ be an oriented closed connected $n$-manifold, let
  $\Gamma := \pi_1(M)$, and let $\alpha \colon \Gamma \actson X$ be a
  standard action. Let $A \subset X$ be a measurable subset
  with~$\Gamma \cdot A = A$ and let $\overline A := X \setminus A$; then the
  restrictions~$\alpha|_A \colon \Gamma \actson A$ and $\alpha|_{\overline A}
  \colon \Gamma \actson \overline A$ are standard actions. If $c, \overline c
  \in C_*(M; \bz \alpha)$ are fundamental cycles, then also
  \[ \chi_A \cdot c + \chi_{\overline A} \cdot \overline c
  \in C_*\bigl(M;\bz \alpha\bigr)
  \]
  is a fundamental cycle. Here, $\chi_A \cdot c$ and
  $\chi_{\overline A} \cdot \overline{c}$ come from the $\bz \alpha^\Gamma$-$\Z \Gamma$-bimodule
  structure on~$\bz \alpha$; more explicitly, if $c=\sum_{i=1}^kf_i\otimes\sigma_i$ is a
  strict chain, then $\chi_A\cdot c$ denotes the strict chain
  $\sum_{i=1}^k(\chi_A \cdot f_i)\otimes\sigma_i=\sum_{i=1}^kf_i|_A\otimes\sigma_i;
  $
  this is well-defined because $A$ satisfies~$\Gamma \cdot A = A$. 
  Indeed the mutually inverse isomorphisms 
  \begin{align*}
    \bz \alpha &\leftrightarrow\bz {\alpha|_A} \oplus\bz {\alpha|_{\overline A}}\\
    f&\mapsto(f|_A,f|_{\overline A})\\
    g|_A+h|_{\overline A}&\mapsfrom(g,h)
  \end{align*}
  of $\Z\Gamma$-modules induce an isomorphism
  \[
  C_*\bigl(M;\bz \alpha\bigr)
  \cong C_*\bigl(M;\bz {\alpha|_A}\bigr)
  \oplus C_*\bigl(M;\bz {\alpha|_{\overline A}}\bigr)
  \]
  of $\Z$-chain complexes that is compatible with the inclusions
  of~$C_*(M;\Z)$ as constant chains. Hence, the $\alpha$-parametrised
  strict fundamental cycles of~$M$ correspond to pairs of $\alpha|_A$-
  and $\alpha|_{\overline A}$-parametrised strict fundamental cycles
  of~$M$. In particular, $\chi_A \cdot c + \chi_{\overline A} \cdot \overline c$
  is an $\alpha$-parametrised strict fundamental cycle of~$M$.
\end{rem}  

\section{Ergodic decomposition}\label{sec:ergdec}

We quickly recall a version of the ergodic decomposition theorem,
introduce notation, and collect some basic properties that will
be used in Section~\ref{sec:proof}.

\begin{defi}[ergodicity]
  A standard probability action $(\alpha,\mu)\colon\Gamma\actson X$ is
  \emph{ergodic} if for every measurable subset~$A\subseteq X$
  with~$\Gamma\cdot A=A$, we have 
  \[
  \mu(A)=0 \qor \mu(A)=1.
  \]
\end{defi}

\begin{defi}[spaces of measures]
  Let $\alpha\colon\Gamma\actson X$ be a standard action;
  \begin{enumerate}
  \item We denote the set of probability measures on~$X$
    by~$\Prob(X)$;
  \item
    We write~$\Prob(\alpha) \subset \Prob(X)$ for the subset of all
    probability measures~$\mu$ on~$X$ for which the action~$\alpha$ is
    $\mu$-preserving.
  \item 
    We write~$\Erg(\alpha) \subset \Prob(\alpha)$ for the
    subset of all $\alpha$-invariant probability measures
    that are ergodic.
  \end{enumerate}
\end{defi}

\begin{defi}[ergodic decomposition]\label{def:ergdec}
  An \emph{ergodic decomposition} of a standard probability
  action~$(\alpha,\mu)\colon\Gamma\actson X$ is a map~$\beta\colon
  X\to\Erg(\alpha)$ (and we will write~$\beta_x$ for~$\beta(x)$) with
  the following properties:
  \begin{enumerate}
  \item
    For every measurable subset~$A \subset X$ the evaluation map
    \begin{align*}
      X&\to[0,1]\\
      x&\mapsto\beta_x(A)
    \end{align*}
    is measurable;
  \item For every measurable subset~$A\subset X$, we have
    $\mu(A)=\int_X\beta_x(A)\;d\mu(x)$;
  \item For all~$x\in X$ and for all~$\gamma\in\Gamma$, we have
    $\beta_{\gamma\cdot x}=\beta_x$;
  \item For every $\nu\in\Erg(\alpha)$, the preimage
    $X_{\nu}:=\beta^{-1}(\nu)$ is measurable and~$\nu(X_{\nu})=1$.
  \end{enumerate} 
\end{defi}

\begin{prop}[integrals and ergodic decomposition]\label{prop:integralsandergdec}
  Let $(\alpha,\mu)\colon\Gamma\actson X$ be a standard probability action 
  and let $\beta\colon X\to\Erg(\alpha)$ be an ergodic decomposition
  of~$(\alpha,\mu)$. For every~$f\in\bz\alpha$, the function
  \begin{align*}
    X&\to\R\\
    x&\mapsto\int_Xf\;d\beta_x
  \end{align*}
  is measurable and
  \[
  \int_Xf \;d\mu=\int_X\biggl(\int_X f\;d\beta_x\biggr)\;d\mu(x).
  \]
\end{prop}
\begin{proof}
  Members of~$\bz\alpha$ are $\Z$-linear combinations of characteristic
  functions of measurable subsets of~$X$. Therefore, linearity of
  integration reduces the claim to the definition of ergodic
  decomposition.
\end{proof}

A standard probability action always admits an ergodic
decomposition. We can even get a stronger existence result: If a standard
action~$\alpha\colon\Gamma\actson X$ admits at least one invariant
probability measure (i.e., $\Prob(\alpha)\ne\varnothing$), then there
exists a \emph{universal ergodic decomposition}, namely a
function~$\beta\colon X\to\Erg(\alpha)$ that is an ergodic decomposition for
every standard probability action~$(\alpha,\mu)\colon\Gamma\actson X$.

\begin{thm}[ergodic decomposition theorem~\protect{\cite[Section~4]{varadarajan}}]
  \label{thm:ergdec}
  Let $\alpha\colon\Gamma\actson X$ be a standard action and let us
  assume that the standard Borel space~$X$ admits at least one
  $\Gamma$-invariant probability measure. Then there exists a
  map~$\beta\colon X\to\Erg(\alpha)$ that for
  every~$\mu\in\Prob(\alpha)$ is an ergodic decomposition
  of~$(\alpha,\mu)\colon\Gamma\actson X$.
\end{thm}

\begin{rem}[\protect{\cite[Lemma~4.1]{varadarajan}}]\label{rem:compatiblemeasure}
  Let $\alpha\colon\Gamma \actson X$ be a measurable action of a countable group
  on a standard Borel space, let $\mu \in \Prob(\alpha)$, and let
  $\beta \colon X \longrightarrow \Erg(\alpha)$ be an ergodic
  decomposition of~$\Gamma \actson (X,\mu)$.

  Let $A \subset X$ be a measurable subset that is 
  $\beta$-compatible in the sense that 
  \[ \fa{x,y \in X} \beta_x = \beta_y
  \Longrightarrow (x\in A \Longleftrightarrow y \in A).
  \]
  Then, we have
  \[ \fa{x \in X} x \in A \Longleftrightarrow \beta_x(A) = 1.
  \]
  
  Indeed, let $x \in A$. From $\beta$-compatibility, we
  obtain that~$X_{\beta_x} \subset A$, and so~$\beta_x(A) \geq
  \beta_x(X_{\beta_x}) =1$. Thus, $\beta_x(A) = 1$.  If $x \in X
  \setminus A$, then we can apply the same argument (with~$A$
  also~$X\setminus A$ is $\beta$-compatible) to show that
  $\beta_x(X\setminus A) =1$, whence~$\beta_x(A) = 0$.
\end{rem}

\begin{rem}\label{rem:essfree}
  Let $\alpha\colon\Gamma \actson X$ be a standard action, let $\mu
  \in \Prob(\alpha)$, and let $\beta \colon X \longrightarrow
  \Erg(\alpha)$ be an ergodic decomposition of~$(\alpha,\mu) \colon
  \Gamma \actson X$.  If the given $\Gamma$-action~$\alpha$ is
  essentially free with respect to~$\mu$, then for $\mu$-almost
  every~$x\in X$, it is also essentially free with respect
  to~$\beta_x$: Indeed, if $\alpha\colon\Gamma\actson X$ is
  essentially free, then $A := \{x\in X\mid\Gamma_x\ne1\}$ is
  measurable and $\mu(A)=0$. The identity
  \[
  0=\mu(A)=\int_X\beta_x(A)\;d\mu(x)
  \]
  yields that $\beta_x(A)=0$ for~$\mu$-almost every $x\in X$.
\end{rem}

\begin{prop}\label{prop:chainergdec}
  Let $M$ be an oriented closed connected manifold with fundamental
  group $\Gamma$. Let $(\alpha,\mu)\colon\Gamma\actson X$ be a
  standard probability action and let $\beta\colon X\to\Erg(\alpha)$ be
  an ergodic decomposition of~$(\alpha,\mu)$. Then, for every strict
  chain~$c\in C_*(M;\bz{\alpha})$, the map
  \begin{align*}
    F_c\colon X&\to\R_{\geq 0}\\
    x&\mapsto\ifsvp{c}{(\alpha,\beta_x)}
  \end{align*}
  is measurable and 
  \[ 
  \ifsvp c {(\alpha,\mu)} = \int_X \ifsvp c {(\alpha,\beta_x)} \;d\mu(x).
  \]
\end{prop}
\begin{proof}
  Measurability of the function~$x\mapsto\ifsvp{c}{(\alpha,\beta_x)}$
  follows from Proposition~\ref{prop:integralsandergdec}. Let
  $c=\sum_{j=1}^m f_j\otimes \sigma_j\in C_*(M;\bz{\alpha})$ be
  in reduced form. Then
  \begin{align*}
    \ifsvp{c}{(\alpha,\mu)}&=\sum_{j=1}^m\int_X\abs{f_j}\;d\mu
    &\text{(definition)}\\
    &=\sum_{j=1}^m\int_X\biggl(\int_X\abs{f_j}\;d\beta_x\biggr)\;d\mu(x)
    & \text{(Proposition~\ref{prop:integralsandergdec})}\\
    &=\int_X\biggl(\sum_{j=1}^m\int_X\abs{f_j}\;d\beta_x\biggr)\;d\mu(x)
    &\text{(linearity)}\\
    &=\int_X\ifsvp{c}{(\alpha,\beta_x)}\;d\mu(x),
  \end{align*}
  as claimed.
\end{proof}

\section{Proof of Theorem~\ref{thm:ifsvergdec}}\label{sec:proof}


Regarding the proof of Theorem~\ref{thm:ifsvergdec}, we have shown 
that the parametrised simplicial volume can be computed via strict
fundamental cycles.  The strict chain
complex~$\bz{\alpha}\otimes_{\Z\Gamma}C_*(\ucov{M};\Z)$ is, in general, not
countable; our goal is now to reduce to a countable
subcomplex~$\Sigma_*(M,\alpha;\Z)$ that works uniformly for all probability
measures on the given standard action.  We proceed in two steps: We
reduce the singular simplices and the measurable function spaces
separately and then combine both reductions through the tensor
product.

\subsection{Reduction: countably many simplices}

We first show that we need only countably many singular simplices:

\begin{prop}\label{prop:simplices}
  Let $M$ be a compact connected manifold with fundamental
  group~$\Gamma$. Then there exists a countable
  $\Z\Gamma$-subcomplex~$C_*$ of~$C_*(\ucov M;\Z)$ such that the
  inclusion~$C_* \hookrightarrow C_*(\ucov M;\Z)$ is a
  $\Z\Gamma$-chain homotopy equivalence and such that there exists a
  chain homotopy inverse of norm at most~$1$.
\end{prop}
\begin{proof}
  We proceed by the standard inductive simplices selection procedure,
  similar to the non-equivariant case of smooth
  simplices~\cite[Lemma~18.9]{leesmooth}: As a compact manifold, $M$
  is homotopy equivalent to a countable (even finite) simplicial
  complex~$M'$~\cite{milnor,siebenmann,kirbysiebenmann}.
  The chain complex~$C_*(\ucov M;\Z)$ is thus $\Z\Gamma$-chain
  homotopy equivalent to~$C_*(\ucov M';\Z)$, where the chain
  homotopies in both directions may be chosen of norm at most~$1$.
  Therefore, we may and will assume that $M$ itself is a countable
  simplicial complex.

  Let $S$ be the set of all singular simplices of~$\ucov M$
  that are (geometric realisations of) simplicial maps to~$\ucov M$,
  defined on iterated barycentric subdivisions of the standard simplices.
  Using inductive (relative) simplicial approximation we can thus
  find a 
  family~$(h_\sigma\colon \Delta^{\dim \sigma} \times [0,1] \to \ucov
  M)_{\sigma \in \map(\Delta^*,\ucov M)}$ of continuous maps with the
  following properties:
  \begin{itemize}
  \item The set~$S$ is closed under taking faces and under the
    deck transformation action.
  \item For all~$\sigma \in \map(\Delta^*, \ucov M)$, we have~$h_{\sigma}(\args,1) \in S$.
  \item If $\sigma \in S$, then $h_\sigma$ is the constant homotopy
    from~$\sigma$ to itself.
  \item For all~$\sigma \in \map(\Delta^*,\ucov M)$ and all~$\gamma \in \Gamma$,
    we have~$h_{\gamma \cdot \sigma} = \gamma \cdot h_\sigma$.
  \item For all~$n \in \N$, all~$j \in \{0,\dots, n\}$, and
    every singular $n$-simplex~$\sigma \in \map(\Delta^n,\ucov M)$, we have
    \[ h_\sigma \circ (\iota_j \times \id_{[0,1]})
    = h_{\sigma \circ \iota_j}
    \]
    where $\iota_j \colon \Delta^{n-1} \to \Delta^n$ denotes the
    inclusion of the $j$-th face.
  \end{itemize}
  We define~$C_*$ to be the subcomplex of~$C_*(\ucov M;\Z)$ spanned
  by~$S$, which is countable. The inclusion~$i_* \colon C_*
  \hookrightarrow C_*(\ucov M;\Z)$ is a $\Z\Gamma$-chain map.
  Conversely, we define~$\Phi_* \colon C_*(\ucov M;\Z) \to C_*$ as the
  $\Z\Gamma$-linear extension of
  \[ \fa{\sigma \colon \Delta^*\to\ucov M}
     \Phi_*(\sigma) := h_\sigma(\args,1);
  \]
  this indeed is a chain map.    
  Moreover, $\Phi_* \circ i_* = \id_{C_*}$ and the standard prism
  decomposition of~$\Delta^*\times[0,1]$ shows that
  $i_* \circ \Phi_* \simeq_{\Z\Gamma} \id_{C_*(\ucov M;\Z)}$.
  By construction, $\|\Phi_*\| \leq 1$.
\end{proof}

\begin{cor}\label{cor:integralcycles}
  Let $M$ be an oriented closed connected manifold with
  fundamental group~$\Gamma$ and let $C_*$ be
  as provided by Proposition~\ref{prop:simplices}.
  Then $\Z \otimes_{\Z\Gamma} C_*$ contains an integral
  fundamental cycle of~$M$.
\end{cor}
\begin{proof}
  Let $c \in \Z \otimes_{\Z\Gamma} C_*(\ucov M;\Z)$ be an integral
  fundamental cycle of~$M$ and let 
  $\Phi_*$ be a chain homotopy inverse as provided
  by Proposition~\ref{prop:simplices}. 
  Then 
  $(\id_\Z \otimes_{\Z\Gamma} \Phi_*) (c)$ 
  is an integral fundamental cycle of~$M$ that
  lies in the subcomplex~$\Z \otimes_{\Z\Gamma} C_*$.
\end{proof}

\subsection{Reduction: countably many functions}

Standard Borel spaces have the following uniform
regularity property for probability measures:

\begin{prop}\label{prop:regular}
  Let $X$ be a standard Borel space. Then there exists
  a countable subalgebra~$\Sigma$ of the Borel $\sigma$-algebra 
  of~$X$ that is dense with respect to \emph{every} probability
  measure on~$X$. I.e.: For every probability measure~$\mu$,
  for every measurable subset~$A$ and every~$\varepsilon \in \R_{>0}$,
  there exists an~$A' \in \Sigma$ with
  \[ \mu (A \symmdiff A') < \varepsilon.
  \]
\end{prop}
\begin{proof}
  As $X$ is a standard Borel space, we may assume without loss of
  generality that $X$ is the Borel space associated with a separable
  metric space~$Y$. Let $Y' \subset Y$ be a countable dense
  subset. Then the algebra~$\Sigma$ generated by~$\{ U_r(y) \mid y \in
  Y',\ r \in \Q_{>0}\}$ has the claimed property (here, $U_r(y)$
  denotes the open ball of radius~$r$ and centre in~$y$ with respect
  to the metric of~$Y$):

  Indeed, let $\mu$ be a probability measure on~$X$ and let $A \subset
  Y$ be measurable.  By regularity on standard Borel
  spaces~\cite[Theorem~17.10]{kechris2012classical}, we have
  \[
  \mu(A)=
  \inf\bigl\{\mu(U)
  \bigm| \text{$U \subset X$ is open and~$A \subset U$}\bigr\}.
  \]
  Hence, it suffices to prove the claim if $A$ is open.
  If $A$ is open, then 
  $A=\bigcup_{n\in\N}U_n$, where each of the~$U_n$ is of the
  form~$U_r(y)$ with $y\in Y'$ and~$r\in\Q_{>0}$, because $\{
  U_r(y) \mid y \in Y',\ r \in \Q_{>0}\}$ is a basis for the
  topology on~$Y$.
  Then
  \[ \mu(A) = \lim_{n\to\infty} \mu \Bigl( \bigcup_{j=0}^n U_j \Bigr).
  \]
  So, for every~$\varepsilon \in \R_{>0}$, if $n \in \N$ is large
  enough, then $U := \bigcup_{j=0}^n U_j$ satisfies~$\mu(A \symmdiff
  U) = \mu( A \setminus U) < \varepsilon$; moreover, by
  construction,~$U \in \Sigma$.
\end{proof}

\begin{cor}\label{cor:regular_equivariant}
  Let $\Gamma \actson X$ be a standard action. Then there exists a
  $\Gamma$-invariant countable subalgebra~$\Sigma$ of the Borel
  $\sigma$-algebra on~$X$ that is dense with respect to \emph{every}
  probability measure on~$X$.
\end{cor}
\begin{proof}
  Let $\Sigma$ be a subalgebra as provided by
  Proposition~\ref{prop:regular}.  Then the algebra generated by~$\{
  \gamma \cdot A \mid A \in \Sigma,\ \gamma \in \Gamma \}$ has the
  claimed property.
\end{proof}

\subsection{Reduction: countably many parametrised chains}

We now combine the geometric and the dynamical reduction steps:

\begin{prop}\label{prop:countable}
  Let $M$ be an oriented closed connected $n$-manifold, let $\Gamma
  :=\pi_1(M)$, and let $C_*$ be as provided by
  Proposition~\ref{prop:simplices}.  Moreover, let $\Gamma \actson X$
  be a standard action, let $\Sigma$ be
  as in Corollary~\ref{cor:regular_equivariant}, and let
  $\B_\Sigma(\alpha,\Z) := \spanz \{ \chi_A \mid A \in \Sigma \} \subset
  \bz \alpha$.  Then the chain complex
  \[ \Sigma_*(M,\alpha;\Z)
  := \B_\Sigma(\alpha,\Z) \otimes_{\Z \Gamma} C_*
  \]
  has the following property: For every $\alpha$-invariant probability
  measure~$\nu$ on~$X$, we have 
  \[ \ifsvp M {(\alpha,\nu)}
  = \inf
    \bigl\{ \ifsvp c {(\alpha,\nu)}
    \bigm| c \in \Sigma_*(M,\alpha;\Z)
    \text{\ is a strict fundamental cycle}
    \bigr\}.
  \]
\end{prop}
\begin{proof}
  By Proposition~\ref{prop:strict}, we know that $\ifsvp M
  {(\alpha,\nu)}$ can be computed by fundamental cycles in~$\bz \alpha
  \otimes_{\Z \Gamma} C_*(\ucov M;\Z)$. We split the argument into
  the following steps:
  \[ \Sigma_*(M,\alpha;\Z)
  \hookrightarrow \bz \alpha \otimes_{\Z \Gamma} C_*
  \hookrightarrow \bz \alpha \otimes_{\Z \Gamma} C_*(\ucov M;\Z).
  \]
  First, we have
  \[ \ifsvp M {(\alpha,\nu)}
  = \inf
    \bigl\{ \ifsvp c {(\alpha,\nu)}
    \bigm| c \in \bz \alpha \otimes_{\Z \Gamma} C_*
    \text{\ is a fundamental cycle}
    \bigr\},
  \]
  because:
  Clearly, we have~``$\leq$''.
  Conversely, let $\Phi_* \colon C_*(\ucov{M};\Z)\to C_*$ be a $\Z\Gamma$-chain
  homotopy inverse of the inclusion~$C_*\hookrightarrow
  C_*(\ucov{M};\Z)$ with $\lVert\Phi_*\rVert \leq 1$, as provided by
  Proposition~\ref{prop:simplices}.  Then
  \[
  \id_{\bz{\alpha}}\otimes_{\Z\Gamma}\Phi_*
  \colon \bz{\alpha}\otimes_{\Z\Gamma}C_*(\ucov{M};\Z)\to \bz{\alpha}\otimes_{\Z\Gamma}C_*
  \]
  is a chain homotopy inverse of
  $\bz{\alpha}\otimes_{\Z\Gamma}C_*\hookrightarrow
  \bz{\alpha}\otimes_{\Z\Gamma}C_*(\ucov{M};\Z)$ that is compatible with
  integral chains and thus maps fundamental cycles to fundamental
  cycles. Moreover, $\lVert\id_{\bz{\alpha}}\otimes_{\Z\Gamma}\Phi_*\rVert \le1$.
  Therefore, ``$\geq$'' holds as well.
  	
  Second, we have
  \[ \ifsvp M {(\alpha,\nu)}
  = \inf
    \bigl\{ \ifsvp c {(\alpha,\nu)}
    \bigm| c \in \Sigma_*(M,\alpha;\Z) 
    \text{\ is a fundamental cycle}
    \bigr\},
  \]
  because:
  Again, ``$\leq$'' is clear. For the converse estimate, 
  let $c\in\bz{\alpha}\otimes_{\Z\Gamma}C_n$ be a fundamental cycle and let
  $\varepsilon \in \R_{>0}$. We construct a fundamental cycle~$c' \in
  \Sigma_n(M,\alpha;\Z)$ with~$\ifsvp{c' - c}{(\alpha,\nu)} \leq \varepsilon$. 
  As $C_*$ is a chain complex and $c$ is a fundamental cycle, $c$ is
  of the form~$c=c_{\Z}+\partial d$ for some integral fundamental
  cycle~$c_\Z\in \Z\otimes_{\Z\Gamma}C_n\hookrightarrow\bz{\alpha}\otimes_{\Z\Gamma}C_n$ and
  some~$d \in \bz{\alpha}\otimes_{\Z\Gamma}C_{n+1}$. Using
  the density result of Proposition~\ref{prop:regular}, we obtain that
  $\B_\Sigma(\alpha,\Z)$ is dense in~$\bz\alpha$ with respect to the $L^1$-norm
  induced by~$\nu$. Hence, also $\B_\Sigma(\alpha,\Z)
  \otimes_{\Z\Gamma} C_*$ is dense in~$\bz \alpha \otimes_{\Z\Gamma}
  C_*$. In particular, there exists a chain~$d' \in \B_\Sigma(\alpha,\Z)
  \otimes_{\Z\Gamma} C_{n+1}$ with~$\ifsvp {d' -d} {(\alpha,\nu)} \leq
  \varepsilon/(n+2)$. Then
  \[ 
  c' := c_\Z + \partial d'
  \]
  is a fundamental cycle in~$\Sigma_n(M,\alpha;\Z)$ and
  \[ 
  \ifsvp{c' -c}{(\alpha,\nu)}
  = \ifsvp{\partial (d' - d)}{(\alpha,\nu)}
  \leq (n+2) \cdot \ifsvp{d' -d}{(\alpha,\nu)}
  \leq \varepsilon,
  \]
  as desired.
\end{proof}

\subsection{Proof of the ergodic decomposition formula}

We prove Theorem~\ref{thm:ifsvergdec} using the countable setting in
Proposition~\ref{prop:countable}.  We first establish notation:

Let $M$ be an oriented closed connected $n$-manifold, let $\Gamma
:=\pi_1(M)$ and let $C_*$ be as provided by
Proposition~\ref{prop:simplices}. Let $\alpha \colon \Gamma \actson X$
be a standard action and let $\Sigma_*(M,\alpha;\Z)$ be the chain complex
provided by Proposition~\ref{prop:countable}.  Let~$\fc (M,\alpha) \subset
\Sigma_n(M,\alpha;\Z)$ be the set of all fundamental cycles of~$M$
in~$\Sigma_*(M,\alpha;\Z)$.

We use an integral fundamental cycle as baseline: There exists an
integral fundamental cycle~$c_\Z \in \Z \otimes_{\Z\Gamma} C_*$
(Corollary~\ref{cor:integralcycles}).  As $\B_\Sigma(\alpha,\Z)$ contains
all constant $\Z$-valued functions, we can view~$c_\Z$ as a
fundamental cycle in~$\Sigma_*(M,\alpha;\Z)$.  We set~$v := |c_\Z|_1$.

Let $\mu$ be a $\Gamma$-invariant probability measure on~$X$
and let $\beta \colon X \longrightarrow \Erg(\alpha)$ be an ergodic
decomposition of~$\Gamma \actson (X,\mu)$. For~$c \in
\Sigma_*(M,\alpha;\Z)$, we recall that
\begin{align*}
  F_c \colon X & \longrightarrow \R_{\geq 0} \\
  x & \longmapsto \ifsvp c {(\alpha,\beta_x)}
\end{align*}
is an integrable function
(Proposition~\ref{prop:chainergdec}).
By Proposition~\ref{prop:countable}, for all~$x \in X$,
we have
\begin{align*}
  \ifsvp M {(\alpha,\beta_x)}
= \inf_{c \in \fc (M,\alpha)} F_c(x) .
\end{align*}
We emphasise that $\fc(M,\alpha)$ is countable.
As a countable infimum of integrable functions 
the function
\begin{align*}
	F \colon X & \longrightarrow \R_{\geq 0} \\
	x & \longmapsto \inf_{c\in \fc(M,\alpha)} F_c(x) = \ifsvp M {(\alpha,\beta_x)}
\end{align*}
is measurable and bounded by~$\Vert M \rVert_\Z$;
hence, $F$ is $\mu$-integrable and 
\begin{align*}
	\int_X \ifsvp M {(\alpha,\beta_x)} \;d\mu(x)
	& = \int_X F \; d\mu
	& \text{(definition of~$F$)}
        \\
        & = \ifsvp M {(\alpha,\mu)}.
        & \text{(Lemma~\ref{lem:selfref} below)}
\end{align*}        

It remains to show the second equality, i.e., that we
can indeed swap taking this specific infimum with integration.

\begin{lem}\label{lem:selfref}
  In the situation above, we have
  \[ \int_X  F \; d\mu
  = \ifsvp M {(\alpha,\mu)}.
  \]
\end{lem}
\begin{proof}
  On the one hand, by definition, $F \leq F_c$
  for all~$c \in \fc(M,\alpha)$. In particular,
  monotonicity of the integral gives
  \begin{align*} \int_X F \; d\mu
    & \leq \inf_{c \in \fc(M,\alpha)} \int_X F_c \;d\mu.
    \\
    & = \inf_{c \in \fc(M,\alpha)}  \ifsvp{c}{(\alpha,\mu)} 
    & \text{(Proposition~\ref{prop:chainergdec})}
    \\
    & = \ifsvp M {(\alpha,\mu)}.
    & \text{(Proposition~\ref{prop:countable})}
  \end{align*}

  For the converse inequality, we use the ``self-referentiality'' of
  the construction to produce fundamental cycles in~$C_*(M;\bz \alpha)$
  with ``small'' norm: For notational convenience, we enumerate
  the countable set~$\fc(M,\alpha)$ as~$c_0, c_1, \dots$.  Let
  $\varepsilon \in \R_{>0}$. For~$n \in \N$, we consider the set
  \[ A_n :=
  \bigl\{ x \in X
  \bigm| \ifsvp {c_n}{(\alpha,\beta_x)}
  \leq \ifsvp M {(\alpha,\beta_x)} + \varepsilon
  \bigr\}.
  \]
  Then, $A_n$ is measurable and $\Gamma \cdot A_n = A_n$.
  Moreover, $\bigcup_{n \in \N} A_n  =X$ by Proposition~\ref{prop:countable}.
  Let $N\in \N$ be so large that
  $ A := \bigcup_{n=0}^N A_n$ satisfies~$\mu(X \setminus A) \leq \varepsilon$.
  We then consider the pairwise disjoint family~$(\overline A_n)_{n \in \{0,\dots,N\}}$
  given by~$\overline A_0 := A_0$ and
  \[ \fa{n \in \{0,\dots, N-1\}}
  \overline A_{n+1}
  := A_{n+1} \setminus \bigcup_{j=0}^n \overline A_j.
  \]
  Then, also the~$\overline A_n$ are measurable and satisfy~$\Gamma
  \cdot \overline A_n = \overline A_n$.
  By construction, $A = \bigsqcup_{n=0}^N \overline A_n$ and 
  \[ c:= \sum_{n=0}^N \chi_{\overline A_n} \cdot c_n + \chi_{X\setminus A} \cdot c_\Z
  \in C_*\bigl( M; \bz \alpha \bigr)
  \]
  is a strict fundamental cycle of~$M$ (Remark~\ref{rem:decompinv}).
  We show that $\ifsvp c {(\alpha,\mu)}$ is small enough:
  By construction, the sets~$\overline A_n$ are $\beta$-compatible.
  For all~$n \in \{0,\dots,N\}$ and all~$x \in \overline A_n$,
  we thus have~$\beta_x(\overline A_n) = 1$ (Remark~\ref{rem:compatiblemeasure})
  and so
  \[ F_c(x) = \ifsvp{c_n}{(\alpha,\beta_x)}
  \leq \ifsvp M {(\alpha,\beta_x)} + \varepsilon
  \leq F(x) + \varepsilon;
  \]
  furthermore, 
  $F_c|_{X\setminus A}=\abs{c_\Z}_1$. We conclude
  that
  \begin{align*}
    \ifsvp M {(\alpha,\mu)}
    & \leq 
    \ifsvp c {(\alpha,\mu)}
    = 
    \int_X  F_c \; d\mu
    = \sum_{n=0}^N \int_{\overline A_n}  F_c \; d\mu
    + \int_{X \setminus A}  F_c \; d\mu
    \\
    & \leq \sum_{n = 0}^N \int_{\overline A_n}  F + \varepsilon \; d\mu
    + \int_{X \setminus A} v \; d\mu
    \\
    & \leq \sum_{n=0}^N \int_{\overline A_n}  F\; d\mu
    + \sum_{n=0}^N \mu(\overline A_n) \cdot \varepsilon
    + \mu(X \setminus A) \cdot v
    \\
    & \leq \int_X  F \; d\mu + \mu (A) \cdot \varepsilon + \varepsilon \cdot v
    \leq \int_X  F \; d\mu + \varepsilon \cdot (1 + v).
  \end{align*}
  Therefore, taking~$\varepsilon \to 0$ shows
  that $\ifsvp M {(\alpha,\mu)} \leq \int_X F \; d\mu$.
\end{proof}

This finishes the proof of Theorem~\ref{thm:ifsvergdec}.

Finally, we mention two straightforward directions of generalisations:

\begin{rem}[general spaces]\label{rem:general}
  The following version of Theorem~\ref{thm:ifsvergdec} also holds:
  Let $M$ be a path-connected topological space that has the homotopy
  type of a countable CW-complex and that admits a universal
  covering. Let $\Gamma := \pi_1(M)$, let $\alpha\colon\Gamma \actson
  X$ be a standard action, and let $\mu \in \Prob(\alpha)$. If $\zeta
  \in H_*(M;\linfz {\alpha,\mu})$ is an integral homology class
  (i.e., coming from~$H_*(M;\Z)$) and $\beta
  \colon X \longrightarrow \Erg(\alpha)$ is an ergodic decomposition
  of~$(\alpha,\mu) \colon \Gamma \actson X$, then
  \[ \ifsvp \zeta {(\alpha,\mu)}
  = \int_X \ifsvp \zeta {(\alpha,\beta_x)} \;d\mu(x).
  \]
  Indeed, all of our proofs work on this level of generality. We
  restricted ourselves to the manifold/fundamental class case to keep
  the notation somewhat lighter.
\end{rem}

\begin{rem}[general decompositions]
  Moreover, decomposition formulas as in Theorem~\ref{thm:ifsvergdec}
  and Remark~\ref{rem:general} also hold for other decompositions of
  the base measure (as in Definition~\ref{def:ergdec}, but with maps
  to~$\Prob(\alpha)$ instead of~$\Erg(\alpha)$), not only for the
  ergodic decomposition. Indeed, we never used the fact that the
  measures occurring in the decomposition are ergodic.
\end{rem}

\bibliographystyle{abbrv}
\bibliography{bib}

\begin{bibdiv}
\begin{biblist}

\bib{abertnikolov}{article}{
      author={Ab{\'e}rt, Mikl{\'o}s},
      author={Nikolov, Nikolay},
       title={Rank gradient, cost of groups and the rank versus {H}eegaard
  genus problem},
        date={2012},
        ISSN={1435-9855},
     journal={J.~Eur.\ Math.\ Soc.},
      volume={14},
      number={5},
       pages={1657\ndash 1677},
         url={https://doi.org/10.4171/JEMS/344},
}

\bib{campagnolocorro}{article}{
      author={Campagnolo, Caterina},
      author={Corro, Diego},
       title={Integral foliated simplicial volume and circle foliations},
        date={2023},
        ISSN={1793-5253,1793-7167},
     journal={J.~Topol.\ Anal.},
      volume={15},
      number={1},
       pages={155\ndash 184},
         url={https://doi.org/10.1142/S1793525321500266},
      review={\MR{4548630}},
}

\bib{FLMQ}{article}{
      author={Fauser, D.},
      author={L{\"o}h, C.},
      author={Moraschini, M.},
      author={Quintanilha, J.~P.},
       title={Stable integral simplicial volume of $3$-manifolds},
        date={2021},
     journal={J.~Topol.},
      volume={14},
      number={2},
       pages={608\ndash 640},
}

\bib{fauserS1}{article}{
      author={Fauser, Daniel},
       title={Integral foliated simplicial volume and {$S^1$}-actions},
        date={2021},
     journal={Forum Math.},
      volume={33},
      number={3},
       pages={773\ndash 788},
}

\bib{FLPS}{article}{
      author={Frigerio, R.},
      author={L{\"o}h, C.},
      author={Pagliantini, C.},
      author={Sauer, R.},
       title={Integral foliated simplicial volume of aspherical manifolds},
        date={2016},
     journal={Israel J.\ Math.},
      volume={216},
       pages={707\ndash 751},
}

\bib{gaboriau_cost}{article}{
      author={Gaboriau, Damien},
       title={Co{\^u}t des relations d'{\'e}quivalence et des groupes},
        date={2000},
        ISSN={0020-9910},
     journal={Invent.\ Math.},
      volume={139},
      number={1},
       pages={41\ndash 98},
         url={https://doi.org/10.1007/s002229900019},
}

\bib{gromov_metric}{book}{
      author={Gromov, Misha},
       title={Metric structures for {R}iemannian and non-{R}iemannian spaces},
      series={Progress in Mathematics},
   publisher={Birkh\"{a}user Boston, Inc., Boston, MA},
        date={1999},
      volume={152},
        ISBN={0-8176-3898-9},
        note={Based on the 1981 French original, with appendices by M.~Katz,
  P.~Pansu and S.~Semmes, translated from the French by Sean Michael Bates},
}

\bib{kechris2012classical}{book}{
      author={Kechris, Alexander},
       title={Classical {D}escriptive {S}et {T}heory},
   publisher={Springer Science \& Business Media},
        date={2012},
      volume={156},
}

\bib{kechris_global}{book}{
      author={Kechris, Alexander~S.},
       title={Global {A}spects of {E}rgodic {G}roup {A}ctions},
      series={Mathematical Surveys and Monographs},
   publisher={American Mathematical Society, Providence, RI},
        date={2010},
      volume={160},
        ISBN={978-0-8218-4894-4},
         url={https://doi.org/10.1090/surv/160},
}

\bib{kechrismiller}{book}{
      author={Kechris, Alexander~S.},
      author={Miller, Benjamin~D.},
       title={Topics in {O}rbit {E}quivalence},
      series={Lecture Notes in Mathematics},
   publisher={Springer-Verlag, Berlin},
        date={2004},
      volume={1852},
        ISBN={3-540-22603-6},
         url={https://doi.org/10.1007/b99421},
}

\bib{kirbysiebenmann}{article}{
      author={Kirby, R.~C.},
      author={Siebenmann, L.~C.},
       title={On the triangulation of manifolds and the {H}auptvermutung},
        date={1969},
        ISSN={0002-9904},
     journal={Bull.\ Amer.\ Math.\ Soc.},
      volume={75},
       pages={742\ndash 749},
         url={https://doi.org/10.1090/S0002-9904-1969-12271-8},
}

\bib{leesmooth}{book}{
      author={Lee, John~M.},
       title={Introduction to smooth manifolds},
     edition={Second},
      series={Graduate Texts in Mathematics},
   publisher={Springer, New York},
        date={2013},
      volume={218},
        ISBN={978-1-4419-9981-8},
}

\bib{loeh_cost}{article}{
      author={L{\"o}h, Clara},
       title={Cost vs.\ integral foliated simplicial volume},
        date={2020},
        ISSN={1661-7207},
     journal={Groups Geom.\ Dyn.},
      volume={14},
      number={3},
       pages={899\ndash 916},
         url={https://doi.org/10.4171/ggd/568},
}

\bib{loehmoraschinisauer}{article}{
      author={L{\"o}h, Clara},
      author={Moraschini, Marco},
      author={Sauer, Roman},
       title={Amenable covers and integral foliated simplicial volume},
        date={2022},
     journal={New York J.\ Math.},
      volume={28},
       pages={1112\ndash 1136},
}

\bib{loehpagliantini}{article}{
      author={L{\"o}h, Clara},
      author={Pagliantini, Cristina},
       title={Integral foliated simplicial volume of hyperbolic
  $3$-man\-i\-folds},
        date={2016},
     journal={Groups Geom. Dyn.},
      volume={10},
       pages={825\ndash 865},
}

\bib{milnor}{article}{
      author={Milnor, John},
       title={On spaces having the homotopy type of a {${\rm CW}$}-complex},
        date={1959},
        ISSN={0002-9947},
     journal={Trans.\ Amer.\ Math.\ Soc.},
      volume={90},
       pages={272\ndash 280},
         url={https://doi.org/10.2307/1993204},
}

\bib{mschmidt}{thesis}{
      author={Schmidt, Marco},
       title={{$L^2$}-{B}etti {N}umbers of $\mathcal{R}$-{S}paces and the
  {I}ntegral {F}oliated {S}implicial {V}olume},
        type={Ph.D. Thesis},
        date={2005},
        note={http://nbn-resolving.de/urn:nbn:de:hbz:6-05699458563},
}

\bib{siebenmann}{article}{
      author={Siebenmann, L.~C.},
       title={On the homotopy type of compact topological manifolds},
        date={1968},
        ISSN={0002-9904},
     journal={Bull.\ Amer.\ Math.\ Soc.},
      volume={74},
       pages={738\ndash 742},
         url={https://doi.org/10.1090/S0002-9904-1968-12022-1},
}

\bib{varadarajan}{article}{
      author={Varadarajan, V.~S.},
       title={Groups of automorphisms of {B}orel spaces},
        date={1963},
        ISSN={0002-9947},
     journal={Trans.\ Amer.\ Math.\ Soc.},
      volume={109},
       pages={191\ndash 220},
         url={https://doi.org/10.2307/1993903},
}

\end{biblist}
\end{bibdiv}

\end{document}